\documentclass[10pt]{article}
\usepackage{amsmath, amssymb, amsfonts, amsthm, epsfig, float, graphicx, sectsty}
\usepackage[all]{xy}
\title{Quantifying local embeddings into finite groups}
\author{Henry Bradford}

\newtheorem{thm}{Theorem}[section]
\newtheorem{lem}[thm]{Lemma}
\newtheorem{propn}[thm]{Proposition}
\newtheorem{coroll}[thm]{Corollary}
\newtheorem{defn}[thm]{Definition}
\newtheorem{ex}[thm]{Example}

\newtheorem{prob}[thm]{Problem}
\newtheorem{rmrk}[thm]{Remark}
\newtheorem{qu}[thm]{Question}

\DeclareMathOperator{\Alt}{Alt}

\DeclareMathOperator{\Cay}{Cay}

\DeclareMathOperator{\fin}{fin}

\DeclareMathOperator{\GL}{GL}
\DeclareMathOperator{\Ham}{Ham}
\DeclareMathOperator{\hyp}{hyp}

\DeclareMathOperator{\im}{im}
\DeclareMathOperator{\lin}{lin}
\DeclareMathOperator{\Schr}{Schr}
\DeclareMathOperator{\SL}{SL}
\DeclareMathOperator{\sof}{sof}

\DeclareMathOperator{\Sym}{Sym}

\begin{document}

\maketitle

\begin{abstract}
We study a function $\mathcal{L}_{\Gamma}$ which quantifies the LEF 
(local embeddability into finite groups) property 
for a finitely generated group $\Gamma$. 
We compute this \emph{LEF growth} function in some examples, 
including certain wreath products. 
We compare LEF growth with the analogous quantitative 
version of residual finiteness, 
and exhibit a family of finitely generated residually finite 
groups which nevertheless admit many more 
local embeddings into finite groups than they do finite quotients. 
Along the way, we give a new proof that 
B.H. Neumann's continuous family of $2$-generated groups 
contains no finitely presented group, 
a result originally due to Baumslag and Miller. 
We compare $\mathcal{L}_{\Gamma}$ with quantitative versions of soficity 
and other metric approximation properties of groups. 
Finally, we show that there exists a ``universal'' function 
which is an upper bound on the LEF growth of any group 
on a given number of generators, and that 
(for non-cyclic groups) any such function is non-computable. 
\end{abstract}

\section{Introduction}

A group $\Gamma$ is \emph{LEF} 
(locally embeddable into finite groups) 
if the multiplication in $\Gamma$ agrees, 
on any finite subset $F$, 
with the multiplication in some finite group $Q_F$; 
a map from $F$ to $Q_F$ witnessing the agreement is called a 
\emph{local embedding}. 
LEF is one of several properties declaring that a 
(generally infinite) group is in some 
sense \emph{well-approximated} by finite groups. 
In recent years there has been great interest in functions which 
\emph{quantify} approximation properties of finitely generated groups. 
There is now a well-established literature 
on quantitative aspects of 
\emph{residual finiteness} \cite{BouRab,BoRaMcR,BoRaCheTim}, 
and a growing understanding of 
quantitative \emph{conjugacy stability} 
\cite{LaLoMc,Peng} and \emph{subgroup separability} \cite{DePe}. 
The \emph{LEF growth} function that is the subject of this paper 
was first defined in \cite{ArzChe} 
(a closely related invariant was independently 
introduced in \cite{BoRaStud}) 
and discussed in the PhD thesis of M. Cavaleri \cite{Cava}, 
but our understanding of its behaviour has been hitherto undeveloped. 

Let $\Gamma$ be a LEF group with finite generating set $S$. 
The \emph{LEF growth function} $\mathcal{L}_{\Gamma} ^S(n)$ 
is defined to be the smallest order of a finite group $Q$ admitting a 
local embedding from the ball $B_S (n)$ of radius $n$ 
in the word metric induced on $\Gamma$ by $S$. 
The dependence of the function $\mathcal{L}_{\Gamma} ^S$ 
on $S$ is slight, so for the rest of the Introduction 
we suppress $S$ from our notation. 
As noted in \cite{ArzChe,BoRaStud}, 
$\mathcal{L}_{\Gamma} (n)$ may be compared 
with the \emph{full residual finiteness growth function} 
$\mathcal{R}_{\Gamma} (n)$ of a residually finite group $\Gamma$, 
which is the minimal order of a finite quotient of 
$\Gamma$ into which $B_S (n)$ injects. 
Every finite quotient of $\Gamma$ which is injective on 
$F \subseteq \Gamma$ clearly restricts to a local embedding of $F$, 
so $\mathcal{R}_{\Gamma}$ is an upper bound 
for $\mathcal{L}_{\Gamma}$, 
for every finitely generated group $\Gamma$. 
It is however far from true that every local embedding 
into a finite group is the restriction of a finite quotient. 
As such, one may expect that, even among residually finite groups, 
there are finitely generated groups whose full residual finiteness 
growth is strictly faster than their LEF growth. 
We show that this is indeed the case, in a rather strong sense. 

\begin{thm} \label{MainThm}
There is a class $\mathcal{C}$ of $2$-generated groups 
and a function $f_0 : \mathbb{N} \rightarrow \mathbb{N}$ 
satisfying the following: 
\begin{itemize}
\item[(i)] For all $f : \mathbb{N} \rightarrow \mathbb{N}$, 
there exists $B_f \in \mathcal{C}$ such that 
$f \preceq \mathcal{R}_{B_f}$;  

\item[(ii)] For all $\Gamma \in \mathcal{C}$, 
$\mathcal{L}_{\Gamma} \preceq f_0$; 

\item[(iii)] $f_0 (n) \preceq \exp (\exp(n))$. 

\end{itemize}
\end{thm}

Here $\prec$ and $\approx$ denote 
comparison of growth functions 
up to rescaling of the argument 
(see Definition \ref{comparisondefn} below). 
The existence of the class $\mathcal{C}$ 
immediately implies a proof of the following. 

\begin{coroll} \label{MainThmCor}
There exist finitely generated residually finite groups $\Gamma$ 
with the property that 
$\mathcal{L}_{\Gamma} \precnapprox \mathcal{R}_{\Gamma}$. 
\end{coroll}

Corollary \ref{MainThmCor} answers a question of 
Bou-Rabee and Studenmund \cite{BoRaStud}. 
They introduce a function which is 
equivalent to our $\mathcal{L}_{\Gamma}$ 
(see Proposition \ref{LEFmarkedgrps} below), 
called by them \emph{geometric full residual finiteness growth}, 
and ask whether it is always equivalent to $\mathcal{R}_{\Gamma}$. 
Corollary \ref{MainThmCor} also gives an affirmative answer to a 
stronger version of Question 66 from \cite{ArzChe}, 
concerning the weakly sofic profile of residually finite groups 
(see Section \ref{MetricApproxSect} below for definitions). 

\begin{coroll} \label{WeakSoficCoroll}
There exist finitely generated residually finite groups $\Gamma$ 
with the property that 
$\mathcal{D}_{\Gamma} \precnapprox \mathcal{R}_{\Gamma}$, 
where $\mathcal{D}_{\Gamma}$ is the weakly sofic profile of $\Gamma$. 
\end{coroll}

Our proof of Theorem \ref{MainThm} is constructive, and is based on 
an argument of Bou-Rabee and Seward \cite{BoRaSewa}. 
They introduce, for any increasing function $f$, 
an explicit $2$-generated residually finite group $B_f$ 
with the property that $f \preceq \mathcal{R}_{B_f}$. 
Thus, $\mathcal{R}_{\Gamma}$ can grow arbitrarily fast 
among $2$-generated groups $\Gamma$. 
By contrast $\mathcal{L}_{\Gamma}$ 
is bounded above by a ``universal function'' 
depending only on the number of generators. 

\begin{thm} \label{UBThm}
For all $d \geq 1$ there exists a function 
$\Theta_d : \mathbb{N} \rightarrow \mathbb{N}$ 
such that for any $d$-generated LEF group $\Gamma$, 
$\mathcal{L}_{\Gamma} \preceq \Theta_d $. 
For $d \geq 2$, no function with this property is 
bounded above by any recursive function. 
\end{thm}

Given Theorem \ref{UBThm} we have a proof of Theorem \ref{MainThm} 
(i) and (ii), with $f_0 = \Theta_2$. 
Alas our proof of Theorem \ref{UBThm} 
does not give an explicit description of the 
growth of the functions $\Theta_d$, 
but a more careful study of local embeddings of 
the $B_f$ yields the following, 
from which we can prove Theorem \ref{MainThm} (iii). 

\begin{thm} \label{BfBoundThm}
For any increasing function $f$, 
$\mathcal{L}_{B_f} (n) \preceq \exp(\exp(n))$. 
\end{thm}

As we shall see, and as was noted in \cite{BoRaStud}, 
$\mathcal{L}_{\Gamma} \approx \mathcal{R}_{\Gamma}$ 
whenever $\Gamma$ is finitely presented. 
Given this, the next observation is immediate from 
Theorem \ref{BfBoundThm} and the result of Bou-Rabee and Seward. 

\begin{coroll} \label{Bfnotfpcoroll}
If $f(n) \npreceq \exp(\exp(n))$ then $B_f$ 
is not finitely presentable. 
\end{coroll}

Bou-Rabee and Seward \cite{BoRaSewa} prove that 
$B_f$ is not finitely presentable under the stronger 
hypothesis that $f$ is not computable. 
Their proof is based on the fact that a 
finitely presented residually finite group has soluble word problem. 
Working along similar lines, 
we show (Proposition \ref{WordProbProp}) 
that a finitely generated 
recursively presented LEF group with soluble word problem 
has computable LEF growth. 

The groups $B_f$ generalize a construction 
due to B.H. Neumann of a continuum of two-generated groups 
$N(d,1)$ \cite{Neum}, 
where $d : \mathbb{N} \rightarrow \mathbb{N}_{\geq 5}$ 
ranges over all increasing odd-valued functions. 
In the course of our investigations 
we also obtain a new proof of the following result, 
which originally appeared in \cite{BaumMill}. 

\begin{thm} \label{NeumFPThm}
The group $N(d,1)$ is never finitely presented. 
\end{thm}

Comparing $\mathcal{L}_{\Gamma}$ and $\mathcal{R}_{\Gamma}$ 
presents a novel tool for proving that a finitely generated 
LEF group $\Gamma$ is not finitely presentable. 
It would be interesting to apply this method 
to other groups for which finite presentability 
is an open problem; especially for 
groups admitting rich natural actions, 
such as subgroups of graphs of free groups 
(see in particular Problems (22) and (23) from \cite{Wise} Section 21), 
or certain subgroups of automorphism groups of free groups. 

Beyond comparing LEF growth with other asymptotic invariants of groups, 
it is natural to ask which functions can arise as the LEF growth function 
of some finitely generated group. 

\begin{thm} \label{GrowthSpectrumThm}
For each of the following nondecreasing functions 
$f:\mathbb{N}\rightarrow\mathbb{N}$, 
there exists a finitely generated LEF group $\Gamma$ such that 
$\mathcal{L}_{\Gamma} \approx f$. 
\begin{itemize}
\item[(i)] $f(n) = \exp (\exp(n))$; 
\item[(ii)] $f(n) = \exp (n^k)$ for each $k \in \mathbb{N}$; 
\item[(iii)] $f(n) = n^k$ for each $k \in \mathbb{N}$. 
\end{itemize}
\end{thm}

Groups with LEF growth which is exponential or polynomial (of given degree) 
are easy to construct using existing results from the literature 
on word growth and full residual finiteness growth. 
The more challenging cases of growth of type $\exp (\exp(n))$ 
or $\exp (n^k)$ (for $k \geq 2$) are constructed using wreath products; 
they follow from our next result. 

\begin{thm} \label{WPMainThm}
Let $\Gamma$ be a finitely generated LEF group 
with word-growth function $\gamma_{\Gamma}$ and let $\Delta$ 
be a finite nontrivial group with trivial centre. Then: 
\begin{equation*}
\exp \big( \gamma_{\Gamma} (n) \big) \preceq \mathcal{L}_{\Delta \wr \Gamma}(n) 
\preceq \exp \big( \mathcal{L}_{\Gamma}(n) \big)\text{.}
\end{equation*}
\end{thm}

The second inequality of Theorem \ref{WPMainThm} is proved in \cite{ArzChe}; 
to prove the first we observe that it may be verified locally 
that a family of finite centreless groups generate their direct sum. 

The paper is structured as follows: 
in Section \ref{PrelimSect} we formally define the LEF growth 
function; give a characterization in terms of the space 
of marked finitely generated groups; 
discuss the relationships between LEF growth; full residual 
finiteness growth;  word growth and finite presentability. 
In Section \ref{WPSect} we give bounds on the LEF growth of wreath products 
and prove Theorem \ref{GrowthSpectrumThm}. 
In Section \ref{MainThmSect} we prove Theorem \ref{MainThm}; 
Theorem \ref{BfBoundThm} and Theorem \ref{NeumFPThm}. 
In Section \ref{MetricApproxSect} we connect LEF growth 
to the quantitative metric approximations of groups 
developed in \cite{ArzChe} and prove Corollary \ref{WeakSoficCoroll}. 
In Section \ref{univubsect} we prove Theorem \ref{UBThm}. 
The paper concludes with a selection of open problems, 
speculations and suggested directions for future inquiry. 

\subsection*{Basic notation and terminology}

An action of a group $\Gamma$ on a set $\Omega$ will be denoted 
$\Gamma \curvearrowright \Omega$. 
Unless specified all our actions will be on the right. 
For $S \subseteq \Gamma$ the associated \emph{Schreier graph} 
$\Schr (\Gamma,\Omega,S)$ has vertices $\Omega$ 
and edges $\Omega \times S$, with 
$\iota (\omega , s) = \omega$ and $\tau (\omega , s)=\omega s$. 
For $n \in \mathbb{N}$ and $\omega \in \Omega$, 
$B_{S,\omega}(n) \subseteq \Omega$ denotes the closed ball of radius 
$n$ around $\omega$ in the path-metric on $\Schr (\Gamma,\Omega,S)$. 
We shall also denote by $B_{S,\omega}(n)$ 
the induced subgraph of $\Schr (\Gamma,\Omega,S)$ on this set. 
In the special case of $\Omega = \Gamma$ with the regular action, 
$\Schr (\Gamma,\Omega,S)=\Cay(\Gamma,S)$ is the \emph{Cayley graph} 
and we write $B_S(n)$ for $B_{S,1}(n)$. 

\section{Preliminaries} \label{PrelimSect}

\subsection{Definition and first properties}

\begin{defn}
For $\Gamma,\Delta$ groups and $F \subseteq \Gamma$, 
a \emph{partial homomorphism} of $F$ into $\Delta$ is a 
function $\pi : F \rightarrow \Delta$ 
such that, for all $g,h \in F$, if $gh \in F$, 
then $\pi(gh)=\pi(g)\pi(h)$. 
$\pi$ is called a \emph{local embedding} of $F$ into $\Delta$ if it is injective. 
$\Gamma$ is \emph{locally embeddable into finite groups (LEF)} 
if, for all finite $F \subseteq \Gamma$, 
there exists a finite group $Q$ and a local embedding of 
$F$ into $Q$. 
\end{defn}

\begin{rmrk} \label{restrmrk}
If $F^{\prime} \subseteq F \subseteq \Gamma$ 
and $\pi : F \rightarrow \Delta$ is a local embedding, 
then so is $\pi \mid_{F^{\prime}} : F^{\prime} \rightarrow \Delta$. 
\end{rmrk}

Henceforth suppose that $\Gamma$ is generated by the finite set $S$. 
\begin{defn}

The \emph{LEF growth} of $\Gamma$ (with respect to $S$) is: 
\begin{center}
$\mathcal{L}_{\Gamma} ^S (n) = \min \lbrace \lvert Q \rvert 
: \exists \pi :B_S(n)\rightarrow Q\text{ a local embedding} \rbrace$
\end{center}
(with $\mathcal{L}_{\Gamma} ^S (n) = \infty$ 
if the set on the right-hand side is empty). 
\end{defn}

\begin{rmrk} \label{LEFgrwthCrmrk}
\normalfont
It is also natural to define LEF growth relative to a class. 
That is, let $\mathcal{C}$ be a class of finite groups. 
We define $\mathcal{L}_{\Gamma,\mathcal{C}} ^S (n)$ to 
be the minimal order of a group $Q \in \mathcal{C}$ 
such that there is a local embedding $\pi :B_S(n)\rightarrow Q$ 
with $\langle \pi(B_S(n)) \rangle = Q$. 
Thus $\mathcal{L}_{\Gamma} ^S 
= \mathcal{L}_{\Gamma,\mathcal{S}} ^S (n)$ 
for $\mathcal{S}$ the class of all finite groups 
(compare with \cite{ArzChe} Definition 27, 
where the requirement that $\pi(B_S(n))$ generates $Q$ is dropped; 
the generation assumption is more natural in our setting, 
for instance in the context of Proposition \ref{LEFRFpropn} below). 
\end{rmrk}

\begin{defn} \label{comparisondefn}
\normalfont
Given functions $f_1 , f_2 : (0,\infty) \rightarrow (0,\infty)$ 
we write $f_1 \preceq f_2$ if there exists a constant $C>0$ 
such that $f_1(x) \leq f_2 (Cx)$ for all $x \in (0,\infty)$. 
We say \emph{$f_1$ is equivalent to $f_2$} and write 
$f_1 \approx f_2$ if $f_1 \preceq f_2$ 
and $f_2 \preceq f_1$. 
We may also compare functions which are defined only on 
$\mathbb{N}$ under $\preceq$ and $\approx$: 
for this purpose $f : \mathbb{N}\rightarrow (0,\infty)$ 
may be extended to $(0,\infty)$ by declaring 
$f$ to be constant on each interval $(n,n+1]$. 
\end{defn}

\begin{lem}[\cite{ArzChe} Proposition 34 (iii)] \label{subgrplem}
Let $\Delta \leq \Gamma$ be finitely generated by $S^{\prime}$. 
Let $M \in \mathbb{N}$ be such that $S^{\prime} \subseteq B_S (M)$. 
Then for all $n$, $\mathcal{L}_{\Delta} ^{S^{\prime}} (n) \leq \mathcal{L}_{\Gamma} ^S (Mn)$. 
In particular $\mathcal{L}_{\Delta} ^{S^{\prime}} \preceq \mathcal{L}_{\Gamma} ^S$. 
\end{lem}

Consequently, the LEF growth only depends on a choice of 
generating set up to equivalence. 

\begin{coroll}
Let $S,S^{\prime} \subseteq \Gamma$ be finite generating sets. 
Let $L,M \in \mathbb{N}$ be such that $S^{\prime} \subseteq B_S (L)$ 
and $S \subseteq B_{S^{\prime}} (M)$. 
Then: 
\begin{center}
$\mathcal{L}_{\Gamma} ^S (n) \leq \mathcal{L}_{\Gamma} ^{S^{\prime}} (Mn) \leq \mathcal{L}_{\Gamma} ^S (LMn)$. 
\end{center}
In particular 
$\mathcal{L}_{\Gamma} ^S \approx \mathcal{L}_{\Gamma} ^{S^{\prime}}$. 
\end{coroll}

\begin{propn}[\cite{ArzChe} \textsection 5.2] \label{directprodprop}
Let $\Gamma$ and $\Delta$ be LEF groups, 
finitely generated by $S$ and $S^{\prime}$, respectively. 
Then $\Gamma \times \Delta$ is LEF and: 
\begin{center}
$\max \lbrace \mathcal{L}_{\Gamma} ^S(n),
\mathcal{L}_{\Delta} ^{S^{\prime}}(n) \rbrace \leq 
\mathcal{L}_{\Gamma \times \Delta} ^{S \cup S^{\prime}} (n)
\leq \mathcal{L}_{\Gamma} ^S(n)\mathcal{L}_{\Delta} ^{S^{\prime}}(n)$. 
\end{center}
\end{propn}

\subsection{The space of marked groups} \label{markedsubsect}

We recall some basic notions relating to based edge-coloured graphs. 
A \emph{based graph} is a pair $(G,v)$, 
where $G$ is a graph and $v \in V(G)$. 
A morphism $(G_1,v_1)\rightarrow (G_2,v_2)$ is 
a graph-morphism $\phi: G_1 \rightarrow G_2$ 
with $\phi(v_1)=v_2$. 
For $C$ a set, an \emph{edge-colouring} of the graph $G$ in $C$ 
is a function $c:E(G)\rightarrow C$. 
The next example gives the main construction of edge-colourings 
considered here. 

\begin{ex}
Let $\Gamma$ be a group; $\Omega$ be a $\Gamma$-set, 
and $S \subseteq \Gamma$. 
Impose an ordering on the elements of $S$, 
to obtain an ordered $\lvert S \rvert$-tuple 
$\mathbf{S} \in \Gamma ^{\lvert S \rvert}$
(equivalently fix a bijection 
$c:S\rightarrow \lbrace 1,\ldots,\lvert S \rvert \rbrace$). 
Then $\Schr(\Gamma,\Omega,\mathbf{S})$ is naturally 
an edge-coloured graph, via $c\big( (\omega, s) \big) = c(s)$. 
\end{ex}

\begin{rmrk} \label{Schrcolourrmrk}
For each $\omega \in \Omega$ and $1\leq i\leq \lvert S \rvert$, 
there is exactly one $i$-coloured edge in 
$\Schr(\Gamma,\Omega,\mathbf{S})$ leaving $\omega$, 
and exactly one entering it. 
The Cayley graph $\Cay(\Gamma,\mathbf{S})$ is precisely 
$\Schr(\Gamma,\Gamma,\mathbf{S})$, 
with $\Gamma$ acting on itself by left-multiplication. 
\end{rmrk}

For a fixed set $C$, a (based) (iso)morphism of edge-coloured 
graphs with colours in $C$ is just a(n iso)morphism 
of the underlying (based) graphs which preserves the colouring. 

The space $\mathcal{G}_d$ of marked $d$-generated groups 
was introduced in \cite{Grig} and is constructed as follows. 
Fix $d \in \mathbb{N}$. 
Let $\hat{\mathcal{G}}_d$ be the set of pairs $(\Gamma,\mathbf{S})$, 
where $\mathbf{S} \in \Gamma^d$ is an ordered $d$-tuple 
which generates $\Gamma$. 
Define an equivalence relation on $\hat{\mathcal{G}}_d$ 
by declaring $(\Gamma,\mathbf{S}) \sim (\Delta,\mathbf{T})$ 
if the map $s_i \mapsto t_i$ extends to an isomorphism from 
$\Gamma$ to $\Delta$, 
and set $\mathcal{G}_d = \hat{\mathcal{G}}_d / \sim$. 

We may give $\mathcal{G}_d$ the structure of a metric space, 
as follows. We write: 
\begin{center}
$\nu \big( (\Gamma,\mathbf{S}),(\Delta,\mathbf{T}) \big) 
= \max\lbrace n \in \mathbb{N} : 
(B_{\mathbf{S}}(n),e_{\Gamma}) \cong (B_{\mathbf{T}}(n),e_{\Delta}) \rbrace$
\end{center}
(the isomorphism being of based edge-coloured graphs, 
taking the induced subgraph on $B_{\mathbf{S}}(n)$ in $\Cay(\Gamma,\mathbf{S})$, 
and likewise for $B_{\mathbf{T}}(n)$ in $\Cay(\Delta,\mathbf{T})$) and set: 
\begin{center}
$\hat{d} \big( (\Gamma,\mathbf{a}),(\Delta,\mathbf{b}) \big) 
= 2^{-\nu ( (\Gamma,\mathbf{a}),(\Delta,\mathbf{b}) )}$. 
\end{center}
$\hat{d}$ is a pseudo-metric on $\hat{\mathcal{G}}_d$, 
which descends to a metric $d$ on $\mathcal{G}_d$. 

\begin{propn}[\cite{Grig} Proposition 6.1] \label{cpctmrkdgrps}
With the topology induced by $d$, $\mathcal{G}_d$ is compact. 
\end{propn}

Vershik and Gordon \cite{VerGor} observe that a $d$-generated 
group is LEF if and only if it is the limit of finite 
groups in $\mathcal{G}_d$. 
This equivalence may be quantified by considering the 
rate of convergence in the metric $d$, as follows. 

\begin{propn} \label{LEFmarkedgrps}
Let $\Gamma$ be a LEF group, 
generated by the $d$-tuple $\mathbf{S}$. 
Define: 
\begin{center}
$\Phi_{\Gamma} ^{\mathbf{S}}(n) = \min \lbrace \lvert \Delta \rvert 
:\exists \mathbf{T} \in \Delta^d \text{ s.t. }
\nu\big( (\Gamma,\mathbf{S}),(\Delta,\mathbf{T})\big)\geq n\rbrace$.
\end{center}
Then $\Phi_{\Gamma} ^{\mathbf{S}}(n) 
\leq \mathcal{L}_{\Gamma} ^{\mathbf{S}} (n) 
\leq \Phi_{\Gamma} ^{\mathbf{S}}(\lceil 3n/2\rceil)$. 
In particular 
$\Phi_{\Gamma} ^{\mathbf{S}} \approx \mathcal{L}_{\Gamma}$. 
\end{propn} 

A slight generalisation of  $\Phi_{\Gamma} ^{\mathbf{S}}$ 
(taking into account finite subsets other than balls) 
was previously introduced in 
\cite{BoRaStud} under the name of 
\emph{geometric full residual finiteness growth}. 

\begin{lem} \label{LEFmarkedgrpslem1}
Let $\Gamma$ and $\Delta$ be groups, 
let $\mathbf{S}$ be a generating ordered $d$-tuple in $\Gamma$, 
and let $\pi : B_{\mathbf{S}} (n) \rightarrow \Delta$ be a local embedding. 
Then $\nu \big( (\Gamma,\mathbf{S}),(\langle \pi(\mathbf{S}) \rangle,\pi(\mathbf{S})) \big) \geq n$. 
\end{lem}

\begin{proof}
Since $\pi$ is a local embedding, $\im (\pi)=B_{\pi(\mathbf{S})}(n)$. 
For $g \in B_{\mathbf{S}} (n)$ and $1 \leq i \leq d$, 
$g s_i \in B_{\mathbf{S}} (n)$ iff 
$\pi (g s_i) = \pi(g)\pi(s_i) \in B_{\pi(\mathbf{S})}(n)$, 
so $\pi$ induces an isomorphism of based coloured graphs. 
\end{proof}

\begin{lem} \label{LEFmarkedgrpslem2}
Let $\mathbf{S}$ and $\mathbf{T}$ be generating ordered $d$-tuples 
in $\Gamma$ and $\Delta$, respectively. 
Suppose $\nu\big( (\Gamma,\mathbf{S}),(\Delta,\mathbf{T})\big)\geq 3n/2$. 
Then there is a local embedding $B_{\mathbf{S}} (n)\rightarrow\Delta$ 
extending $s_i \mapsto t_i$. 
\end{lem}

\begin{proof}
Let $\phi : (B_{\mathbf{S}}(\lceil 3n/2 \rceil),e)\rightarrow 
(B_{\mathbf{T}}(\lceil 3n/2 \rceil),e)$ be an isomorphism 
of based edge-coloured graphs. Then $\phi(s_i)=t_i$ for all $i$. 
Let $g,h \in B_{\mathbf{S}}(n)$ and suppose $gh\in B_{\mathbf{S}}(n)$. 
Then there exist $0 \leq m_1 , m_2 \leq n$; $1 \leq i_j \leq d$ 
and $\epsilon_j \in \lbrace \pm 1\rbrace$ such that: 
\begin{center}
$g = s_{i_1} ^{\epsilon_1} \cdots s_{i_{m_1}} ^{\epsilon_{m_1}}$ 
and $h = s_{i_{m_1 +1}} ^{\epsilon_{m_1 +1}} \cdots s_{i_{m_1 +m_2}} ^{\epsilon_{m_1 +m_2}}$. 
\end{center}
An easy induction on the word length shows that: 
\begin{center}
$\phi(g) = t_{i_1} ^{\epsilon_1} \cdots t_{i_{m_1}} ^{\epsilon_{m_1}}$ 
and $\phi(h) = t_{i_{m_1 +1}} ^{\epsilon_{m_1 +1}} \cdots t_{i_{m_1 +m_2}} ^{\epsilon_{m_1 +m_2}}$
\end{center}
(since there is exactly one edge of each colour leaving and entering 
each vertex). By the same token, 
\begin{center}
$\phi(gh) = t_{i_{1}} ^{\epsilon_{1}} \cdots t_{i_{m_1 +m_2}} ^{\epsilon_{m_1 +m_2}}
=\phi(g) \phi(h)$
\end{center}
since the corresponding edge-path lies entirely 
within $B_{\mathbf{S}}(\lceil 3n/2 \rceil)$. 
Thus $\phi|_{B_{\mathbf{S}}(n)}$ is a local embedding. 
\end{proof}

\begin{proof}[Proof of Proposition \ref{LEFmarkedgrps}]
Let $(Q_n)$ be a sequence of finite groups such that 
$\lvert Q_n \rvert = \mathcal{L}_{\Gamma} ^{\mathbf{S}} (n)$ 
and there exists a local embedding 
$\pi_n : B_{\mathbf{S}} (n) \rightarrow Q_n$. 
Note that $\pi_n (\mathbf{S})$ generates $Q_n$ 
(else $\pi_n$ is a local embedding into 
$\langle \pi_n (\mathbf{S}) \rangle$ and 
$\lvert\langle \pi_n (\mathbf{S}) \rangle\rvert < \mathcal{L}_{\Gamma} ^{\mathbf{S}} (n)$). 
Then $\Phi_{\Gamma} ^{\mathbf{S}}(n) 
\leq \mathcal{L}_{\Gamma} ^{\mathbf{S}} (n)$ 
by Lemma \ref{LEFmarkedgrpslem1}. 

Conversely, let $(P_n)$ be a sequence of finite groups 
with generating $d$-tuples $\mathbf{S}_n$. 
Suppose $B_{\mathbf{S}} (n) \cong B_{\mathbf{S}_n} (n)$. 
By Lemma \ref{LEFmarkedgrpslem2} 
there is a local embedding 
$\pi_n : B_{\mathbf{S}} (n) \rightarrow P_{\lceil 3n/2\rceil}$, 
so $\mathcal{L}_{\Gamma} ^{\mathbf{S}} (n) 
\leq \lvert P_{\lceil 3n/2\rceil} \rvert$ as required. 
\end{proof}

\subsection{LEF, residual finiteness and finite presentability}

Let $\Gamma$ be a finitely generated group and 
$S \subseteq \Gamma$ be a finite generating set. 
Recall that $\Gamma$ is \emph{residually finite (RF)} if, 
for every finite subset $F \subset \Gamma$, 
there exists a finite group $Q$ and a homomorphism 
$\pi : \Gamma \rightarrow Q$ such that the restriction 
$\pi |_F$ of $\pi$ to $F$ is injective. 
The following quantitative version of this definition 
is introduced in \cite{BoRaMcR} 
(originally under the name \emph{residual girth}). 

\begin{defn}
The \emph{(full) RF growth} of $\Gamma$ is: 
\begin{center}
$\mathcal{R}_{\Gamma} ^S (n) 
= \min \lbrace \lvert F \rvert 
: \exists \pi : \Gamma \rightarrow F \text{ s.t. } 
\pi |_{B_S(n)} \text{ is injective} \rbrace$. 
\end{center}
More generally, for $\mathcal{C}$ a class of finite groups, 
one may define: 
\begin{center}
$\mathcal{R}_{\Gamma,\mathcal{C}} ^S (n) 
= \min \lbrace \lvert F \rvert 
: F\in\mathcal{C}, \exists \pi : \Gamma \twoheadrightarrow F \text{ s.t. } 
\pi |_{B_S(n)} \text{ is injective} \rbrace$. 
\end{center}
\end{defn}

If $Q$ is a finite group and the homomorphism 
$\pi : \Gamma \rightarrow Q$ is injective on the finite subset 
$F \subseteq \Gamma$ then the restriction of $\pi$ to $F$ 
is a local embedding, whence the next inequality 
(see also \cite{ArzChe} \textsection 4.3). 

\begin{lem} 
$\mathcal{L}_{\Gamma} ^S (n) 
\leq \mathcal{R}_{\Gamma} ^S (n)$ for all $n$. 
In particular every RF group is LEF. 
\end{lem}

Already this raises the question of how far apart 
$\mathcal{L}_{\Gamma} ^S (n)$ and $\mathcal{R}_{\Gamma} ^S (n)$ 
can be (the subject of our Theorem \ref{MainThm}). 
The next observation, previously made in \cite{BoRaStud} Lemma 1.3, 
is that for finitely presented groups there is essentially no difference. 
The first part of the following proposition 
is stated as Theorem 2.2 in \cite{VerGor}. 
Their proof (which we reproduce here) also implies the second part. 

\begin{propn} \label{LEFRFpropn}
If $\Gamma$ is finitely presented and LEF then it is residually finite. 
Moreover, if $\Gamma \cong \langle S \mid R \rangle$ 
is a finite presentation with $R \subseteq F(S)$ 
consisting of reduced words of length at most $n_0$, 
and $\mathcal{C}$ is a class of finite groups, 
then: $\mathcal{R}_{\Gamma,\mathcal{C}} ^S (n) 
= \mathcal{L}_{\Gamma,\mathcal{C}} ^S (n)$
for all $n \geq n_0$. 
\end{propn}

The proof is based on the following standard fact 
from combinatorial group theory. 

\begin{lem} \label{extendlem}
Let $\Gamma \cong \langle s_1 , \ldots , s_k \mid R \rangle$ 
be a finite presentation and let $H$ be a group. 
A function $s_i \mapsto h_i$ extends (necessarily uniquely) 
to a homomorphism $\Gamma \rightarrow H$ if and only if, 
for all $w(s_1, \ldots , s_k) \in R$, 
$w(h_1, \ldots , h_k) = 1$. 
\end{lem}

\begin{lem} \label{LEFRFlem}
Let $\Gamma \cong \langle S \mid R \rangle$ and $n_0$ 
be as in Proposition \ref{LEFRFpropn} and let $\Delta$ be a group. 
Then for $n \geq n_0$ any partial homomorphism 
$\pi : B_S(n) \rightarrow \Delta$ extends to a 
homomorphism $\Gamma \rightarrow \Delta$. 
\end{lem}

\begin{proof}
Let $S = \lbrace s_1 , \ldots , s_k \rbrace$ 
and let $\pi (s_i) = t_i \in \Delta$. 
An easy induction on the word length shows that: 
\begin{center}
$\pi (w (s_1 , \ldots , s_k)) = w (t_1 , \ldots , t_k)$
\end{center}
for all $w \in F(S)$ of length at most $n$. 
In particular this holds for $w \in R$. 
However $w (s_1 , \ldots , s_k) = 1$ in $\Gamma$, 
so $w (t_1 , \ldots , t_k) = 1$. 
The result now follows from Lemma \ref{extendlem}. 
\end{proof}

\begin{proof}[Proof of Proposition \ref{LEFRFpropn}]
Apply Lemma \ref{LEFRFlem} to the local embeddings 
$\pi : B_S (n) \rightarrow Q$, for $Q \in \mathcal{C}$. 
\end{proof}

\begin{coroll} 
Suppose $\mathcal{R}_{\Gamma} ^S (n) \neq  \mathcal{L}_{\Gamma} ^S (n)$ for infinitely many $n$. Then $\Gamma$ is not 
finitely presentable. 
\end{coroll}

\begin{coroll} \label{LEFRFequivcoroll}
Suppose $\Gamma$ is finitely presented. 
Then $\mathcal{R}_{\Gamma} \approx \mathcal{L}_{\Gamma}$. 
\end{coroll}

Let $\Phi_{\Gamma} ^{\mathbf{a}}$ be as in the statement 
of Proposition \ref{LEFmarkedgrps}. 
In \cite{BoRaStud} Lemma 1.3 it was proved that for $\Gamma$ 
finitely presented, we have 
$\Phi_{\Gamma} ^{\mathbf{a}} \approx \mathcal{R}_{\Gamma}$. 
As such, Corollary \ref{LEFRFequivcoroll} also follows from 
\cite{BoRaStud} Lemma 1.3 and Proposition \ref{LEFmarkedgrps}. 

\subsection{LEF growth and word growth}

One may also compare LEF growth 
with word growth in finitely generated groups 
(as noted in \cite{ArzChe}). Recall that the \emph{word growth function} 
$\gamma_{\Gamma} ^S$ of $\Gamma$ with respect to $S$ is defined by: 
$\gamma_{\Gamma} ^S (n) = \lvert B_S (n) \rvert$. 
Up to $\approx$, $\gamma_{\Gamma} ^S$ is independent of $S$; 
we may therefore write $\gamma_{\Gamma}$ when to do so 
causes no ambiguity. 

\begin{rmrk}[\cite{ArzChe} \textsection 4.1] \label{growthlowerbdrmrk}
If there is a local embedding $B_S (n) \rightarrow Q$, 
then $\lvert B_S (n) \rvert \leq \lvert Q \rvert$, 
so $\gamma_{\Gamma} ^S (n) \leq \mathcal{L}_{\Gamma} ^S (n)$. 
\end{rmrk}

It is natural to ask when the lower bound 
from Remark \ref{growthlowerbdrmrk} is the only obstruction 
to constructing local embeddings from $\Gamma$ into 
small finite groups. 

\begin{defn}
We describe $\Gamma$ as 
\emph{efficiently locally embeddable in finite groups} (ELEF) 
if $\gamma_{\Gamma} \approx \mathcal{L}_{\Gamma}$. 
\end{defn}

\begin{ex} \label{FreeAbEx}
For any $k \in \mathbb{N}$, $\Gamma = \mathbb{Z}^k$ satisfies 
$n^k \preceq \gamma_{\Gamma}(n) \preceq \mathcal{L}_{\Gamma} (n) 
\preceq \mathcal{R}_{\Gamma} (n) \preceq n^k$, 
so $\mathcal{L}_{\Gamma} (n) \approx n^k$ and $\Gamma$ is ELEF. 
The final inequality $\mathcal{R}_{\Gamma} (n) \preceq n^k$ 
holds since the congruence homomorphism 
$\pi_q : \mathbb{Z}^k \rightarrow (\mathbb{Z}/q\mathbb{Z})^k$ 
restricts to an injection $B_S(n) \rightarrow (\mathbb{Z}/q\mathbb{Z})^k$ 
for $q > n$ (here $S$ is the standard generating set for $\mathbb{Z}^k$). 
\end{ex}

\begin{ex} \label{LinearEx}
\normalfont
Every finitely generated subgroup $\Gamma$ of $\GL_d(\mathbb{Z})$ 
satisfies $\mathcal{L}_{\Gamma} (n) \preceq \mathcal{R}_{\Gamma} (n) \preceq \exp(n)$. 
This is because (as is shown by an easy induction), 
for any finite $S \subseteq \GL_d(\mathbb{Z})$, 
a matrix $g \in B_S (n)$ satisfies 
$\lVert g \rVert_{\infty} \preceq \exp (n)$, 
so that there exists $q \in \mathbb{N}$ with $q \preceq \exp(n)$ 
such that the restriction of the congruence homomorphism 
$\pi_q : \GL_d(\mathbb{Z}) \rightarrow \GL_d(\mathbb{Z}/q\mathbb{Z})$ 
to $B_S (n)$ is injective. 
Thus any $\Gamma \leq \GL_d(\mathbb{Z})$ of exponential growth is ELEF. 
As is well-known, such subgroups are precisely those which are 
not virtually nilpotent. 
\end{ex}

\begin{propn} \label{polyLEFnilpprop}
Let $\Gamma$ be a finitely generated group. 
$\mathcal{L}_{\Gamma}$ is bounded above by a polynomial 
function iff $\Gamma$ is virtually nilpotent. 
Such a group $\Gamma$ is ELEF iff it is virtually abelian. 
\end{propn}

\begin{proof}
If $\mathcal{L}_{\Gamma}$ is polynomially bounded, 
then by Remark \ref{growthlowerbdrmrk} and Gromov's 
polynomial growth theorem, 
$\Gamma$ is virtually nilpotent. 
Conversely recall that every finitely generated virtually nilpotent 
group is finitely presented, 
so by Proposition \ref{LEFRFpropn}, 
$\mathcal{L}_{\Gamma} \approx \mathcal{R}_{\Gamma}$. 
By \cite{BoRaMcR} Theorem 1.3 $\mathcal{R}_{\Gamma}$ 
is polynomially bounded for $\Gamma$ virtually nilpotent. 
For the second claim, by \cite{BoRaStud} Theorem 3 
$\mathcal{R}_{\Gamma} \approx \gamma_{\Gamma}$ 
iff $\Gamma$ is virtually abelian. 
\end{proof}

\section{Wreath products} \label{WPSect}

Recall that, given groups $\Gamma$ and $\Delta$, the 
(restricted, regular) \emph{wreath product} 
of $\Delta$ by $\Gamma$ is defined to be: 
\begin{equation*}
\Delta \wr \Gamma
= \big( \bigoplus_{g \in \Gamma} \Delta \big) \rtimes \Gamma
\end{equation*}
where the action of $\Gamma$ on the \emph{base group} 
$\bigoplus_{g \in \Gamma} \Delta$ is given by: 
\begin{equation*}
\big( (a_h)_{h \in \Gamma}\big)^{g} 
= \big(a_{hg}\big)_{h \in \Gamma} 
\end{equation*}
so that multiplication in $\Delta \wr \Gamma$ is given by: 
\begin{equation} \label{wreathmulteqn}
\big( (a_h)_{h \in \Gamma},g\big)\cdot \big( (b_h)_{h \in \Gamma},k \big) = 
\big( (a_h b_{h g^{-1}})_{h \in \Gamma},gk \big) \text{.}
\end{equation}
For $A \subseteq \Delta$ and $h \in \Gamma$ 
we write $A(h) \subseteq \bigoplus_{g \in \Gamma} \Delta$ 
for the copy of $A$ supported at $h$. 
That is, for $k \in \Gamma$: 
\begin{center}
$\lbrace a_k : a \in A(h) \rbrace
=\left\{ \begin{array}{cc} A & h=k \\ 
\lbrace e \rbrace & \text{otherwise} \end{array} \right.$.
\end{center}

\begin{rmrk} \label{WreathPrelimRmrk}
\begin{itemize}
\item[(i)] If $\Gamma$ and $\Delta$ are finitely generated by $S$ and $T$, respectively, 
then $\Delta \wr \Gamma$ is finitely 
generated by $S \cup T(e)$. 
\item[(ii)] If in addition $\Delta$ is nontrivial and 
$\Gamma$ is infinite, then $\Delta \wr \Gamma$ has exponential word growth. 
\end{itemize}
\end{rmrk}

Gr\"{u}nberg proved that wreath products are very seldom residually finite. 

\begin{thm}[\cite{Gruen} Theorem 3.2] \label{GruenThm}
If $\Delta$ is nonabelian and $\Omega$ is infinite 
then $\Delta \wr \Gamma$ is not residually finite. 
\end{thm}

By contrast Vershik and Gordon \cite{VerGor} showed that if 
$\Gamma$ and $\Delta$ are LEF groups, 
then so is their wreath products $\Delta \wr \Gamma$. 
Roughly speaking, if there are local embeddings from finite subsets 
of $\Gamma$ and $\Delta$ to finite groups $Q$ and $P$, respectively, 
then there is a corresponding finite subset of 
$\Delta \wr \Gamma$ admitting a local embedding into $P \wr Q$. 
Specializing to balls in the word-metric, 
and making this construction effective, one obtains the following. 

\begin{thm}[\cite{ArzChe} Theorem 40] \label{WreathUBThm}
Let $\Gamma$ and $\Delta$ be LEF groups, 
finitely generated by $S$ and $T$, respectively. 
Suppose $e_{\Gamma} \in \Gamma$ and $e_{\Delta} \in \Delta$. 
Then: 
\begin{equation} \label{WPregeqn}
\mathcal{L}_{\Delta \wr \Gamma} ^{T(e)\times S}  (n) 
\leq \big( \mathcal{L}_{\Delta} ^T (4n)^{\mathcal{L}_{\Gamma} ^{S} (4n)} \big) \mathcal{L}_{\Gamma} ^S (4n)\text{.} 
\end{equation}
\end{thm}

The proof in \cite{ArzChe} 
is in turn based on \cite{HoltRees} Theorem 3.1, 
wherein the additional hypothesis is made 
that $\Gamma$ is residually finite, 
however the same proof works for $\Gamma$ LEF; 
the point is that for a group $Q$ and $F \subseteq \Gamma$, a local embedding 
$\phi : F \rightarrow Q$ induces a local embedding: 
\begin{center}
$\langle \Delta (f) : f \in F \rangle \cdot F \rightarrow \Delta \wr Q$. 
\end{center}
At no point in the proof of \cite{HoltRees} Theorem 3.1 
is the hypothesis used that $\phi$ is the restriction of a 
homomorphism. 

It is unclear in general whether the construction, described above, 
of local embeddings 
from $\Delta \wr \Gamma$ based on those from $\Gamma$ and $\Delta$, 
in the most efficient possible in terms of the LEF growth. 
In certain special cases however, we do have a significant obstruction 
to the existence of local embeddings into small finite groups, 
coming from the fact that commuting elements can be recognized locally. 
Combining the next result with Theorem \ref{WreathUBThm}, 
we immediately obtain Theorem \ref{WPMainThm}. 

\begin{thm} \label{WreathLBThm}
Suppose that $\Delta$ is a non-trivial finite group 
satisfying $Z(\Delta) = 1$. Then: 
\begin{equation} \label{WPeqn2}
\exp \big( \log\lvert \Delta\rvert\gamma_{\Gamma} ^S(n) \big) \leq 
\mathcal{L}_{\Delta \wr \Gamma} ^{S \cup T(e)} (n)
\text{.}
\end{equation}
\end{thm}

\begin{proof}
Let $C > 0$ be sufficiently large (to be chosen),
let $H$ be a finite group and 
let $\pi : B_{S \cup T(e)} (Cn) \rightarrow H$ 
be a local embedding. 
Let $C^{\prime} > 0$ be such that $B_T(C^{\prime}) = \Delta$, 
so that $\Delta (e) \subseteq B_{T(e)} (C^{\prime})$, 
and for $g \in B_S(n)$, 
\begin{center}
$\Delta (g) = \Delta (e)^{g^{-1}} 
\subseteq B_{S \cup T(\omega_0)} (2n+C^{\prime})$.
\end{center}
Therefore, for $C > 2 + C^{\prime}$, 
the restriction of $\pi$ to $\Delta (g)$ 
is an injective homomorphism, so: 
\begin{equation} \label{ctrelesseqn1}
\Delta \cong \pi \big(\Delta (g)\big) \leq H\text{.}
\end{equation}
Next, let $g_1 , g_2 \in B_S(n)$ with $g_1 \neq g_2$. 
Then $\Delta (g_1), \Delta (g_2) 
\leq \Delta \wr \Gamma$ are commuting subgroups 
with trivial intersection, so for $C$ sufficiently large 
($C > 8 + 4 C^{\prime}$ suffices), 
the same holds for 
$\pi \big( \Delta (g_1)\big), 
\pi \big( \Delta (g_2)\big) \leq H$. 
Consider $K=\langle\pi\big(\Delta(g)\big):g\in B_S (n)\rangle\leq H$. 
Then each $\pi \big(\Delta (g)\big) \vartriangleleft K$. 
We claim that $K$ is the direct product of 
the $\pi \big(\Delta (g)\big)$, so that by (\ref{ctrelesseqn1}), 
\begin{center}
$\lvert H \rvert \geq \lvert K \rvert 
= \lvert \Delta \rvert^{\lvert B_S (n) \rvert}$
\end{center}
which yields the desired bound. Suppose otherwise. 
Then there exists $g_0 \in B_S (n)$ such that 
$1 \neq \pi \big(\Delta (g_0)\big) \cap L$, where: 
\begin{center}
$L=\langle\pi\big(\Delta (g)\big):g\in B_S (n), g\neq g_0\rbrace \rangle$. 
\end{center}
But $L$ centralises $\pi \big(\Delta (g_0)\big)$, 
so $\pi \big(\Delta (g_0)\big)$ has non-trivial centre, 
contradicting (\ref{ctrelesseqn1}) 
and our hypothesis on $\Delta$. 
\end{proof}

\begin{proof}[Proof of Theorem \ref{GrowthSpectrumThm}]
As in Example \ref{FreeAbEx} we have: 
\begin{equation} \label{FreeAbIneq}
n^k \approx \gamma_{\mathbb{Z}^k} (n) 
\preceq \mathcal{L}_{\mathbb{Z}^k}(n) 
\preceq \mathcal{R}_{\mathbb{Z}^k} (n) 
\approx n^k
\end{equation}
whence (iii). We also have: 
\begin{equation} \label{ESSELLIneq}
\exp(n) \approx \gamma_{\SL_d(\mathbb{Z})} (n) 
\preceq \mathcal{L}_{\SL_d(\mathbb{Z})}(n) 
\preceq \mathcal{R}_{\SL_d(\mathbb{Z})} (n) 
\approx \exp(n)
\end{equation}
for any $d \geq 2$ (as in Example \ref{LinearEx}). 
Now let $\Delta$ be a finite nontrivial centreless group. 
Taking $\Gamma = \SL_d(\mathbb{Z})$ in Theorem \ref{WPMainThm} 
yields (i), by (\ref{ESSELLIneq}), 
whereas taking $\Gamma = \mathbb{Z}^k$ 
yields (ii), by (\ref{FreeAbIneq}). 
\end{proof}

\section{Local embeddings of permutation groups} \label{MainThmSect}

\subsection{Local isomorphisms of Schreier graphs}

\begin{defn}
Let $G_1 , G_2$ be directed edge-coloured graphs 
(with colours in $C$) and let $l \in \mathbb{N}$. 
We say that $G_1$ is 
\emph{locally embedded in $G_2$ at radius $l$} if, 
for every $v \in V(G_1)$, there exists $w \in V(G_2)$ and an 
isomorphism of based coloured graphs $(B_v (l),v) \cong (B_w (l),w)$. 
We say that $G_1$ and $G_2$ are 
\emph{locally isomorphic at radius $l$} if 
each is locally embedded in the other at radius $l$, 
that is, $G_1$ and $G_2$ have the same set of 
isomorphism-types of balls of radius $l$. 
\end{defn}

\begin{lem} \label{permisolem}
For $i=1,2$ let $\Gamma_i$ be a group acting faithfully on a set 
$\Omega_i$, and let $\mathbf{S}_i$ 
be an ordered generating $d$-tuple in $\Gamma_i$. 
Suppose that the Schreier graphs 
$\Schr (\Gamma_i,\Omega_i,\mathbf{S}_i)$ 
are locally colour-isomorphic at some radius at least $\lceil 3l/2 \rceil$. 
Then there is a local embedding 
$B_{\mathbf{S}_1} (l) \rightarrow \Gamma_2$ extending 
$(\mathbf{S}_1)_j \mapsto (\mathbf{S}_2)_j$. 
\end{lem}

\begin{proof}
For $w_1 , w_2 \in F_d$ of length $\leq \lceil 3l/2\rceil$ 
and $i=1,2$, 
\begin{align} \label{ballcolourSchr1}
w_1 (\mathbf{S}_i) = w_2 (\mathbf{S}_i) 
& \Leftrightarrow \forall \omega \in \Omega_i, 
\omega w_1 (\mathbf{S}_i) = \omega w_2 (\mathbf{S}_i)
\end{align} 
(with the implication from right to left being by faithfulness 
of the action). 
By hypothesis, the sets of isomorphism types 
of based coloured balls in the 
$\Schr (\Gamma_i,\Omega_i,\mathbf{S}_i)$ 
are the same, so the right-hand conditions in (\ref{ballcolourSchr1}) 
are equivalent for $i=1,2$. 
In particular, in the notation of Subsection \ref{markedsubsect}, 
$\nu \big( (\Gamma_1,\mathbf{S}_1) 
, (\Gamma_2,\mathbf{S}_2) \big) \geq 3l/2$, 
and the conclusion follows from Lemma \ref{LEFmarkedgrpslem2}. 
\end{proof}

Lemma \ref{permisolem} provides a very effective practical tool 
for constructing local embeddings between groups. 
Here we give just one application, 
which will be key to the proof of Theorem \ref{BfBoundThm}. 

\begin{ex} \label{ncyc3cycex}
\normalfont
Let $\mathbf{S}=(s_1,s_2) \in \Sym (n)^2$, 
with $s_1$ an $n$-cycle and $s_2$ a $3$-cycle, 
and set $\Gamma = \langle s_1,s_2\rangle$. 
Then $\Schr(\Gamma,[n],\mathbf{S})$ 
consists of a $1$-coloured $n$-cycle; 
a $2$-coloured $3$-cycle with vertices $x,y,z \in [n]$ 
and a $2$-coloured loop based at every vertex 
$w \in [n] \setminus \lbrace x,y,z \rbrace$ 
(see Figure \ref{SymSchreierfig} below). 

The isomorphism type of $G = \Schr(\Gamma,[n],\mathbf{S})$ 
as a coloured graph is determined by the lengths of the 
$1$-coloured arcs between $x$, $y$ and $z$, together 
with the orientation of the $3$-cycle 
(clockwise or counterclockwise, assuming the $n$-cycle 
is oriented clockwise). 

Let $l \in \mathbb{N}$ and note that the set of balls in $G$ 
of radius $l$ detects only those lengths of $1$-coloured arcs between 
$x$, $y$ and $z$ which are at most $2l$. 
We therefore call a $1$-coloured arc ``long'' if it has length at least $2l+1$. 
Assuming now that $n \geq 6l+3$, 
at least one of the $1$-coloured arcs between $x$, $y$ and $z$ is long. 
Varying the length of the long arcs, while keeping them long, 
the set of isomorphism types balls of radius $l$ does not change. 
We therefore obtain, for any $m \geq 6l+3$, 
a graph $G^{\prime}$ with $m$ vertices which is locally isomorphic 
to $G$ at radius $l$ (though smaller values of $m$ may also be 
possible if one arc length is much less than $2l+1$). 
See Figure \ref{SymSchreierfig} for an illustration. 

Now, $G^{\prime}$ is the Schreier graph of 
some transitive  permutation group 
$\Gamma^{\prime} \leq \Sym(m)$ with respect to an ordered 
pair of generators $(s^{\prime} _1,s^{\prime} _2) \in \Sym(m)^2$. 
By Lemma \ref{permisolem} there is a local embedding 
$B_{\mathbf{S}} (2l/3) \rightarrow \Gamma^{\prime}$ 
extending $s_i \rightarrow s^{\prime} _i$. 

\begin{figure}[h] \label{SymSchreierfig}
\centering
\includegraphics[scale=0.3]{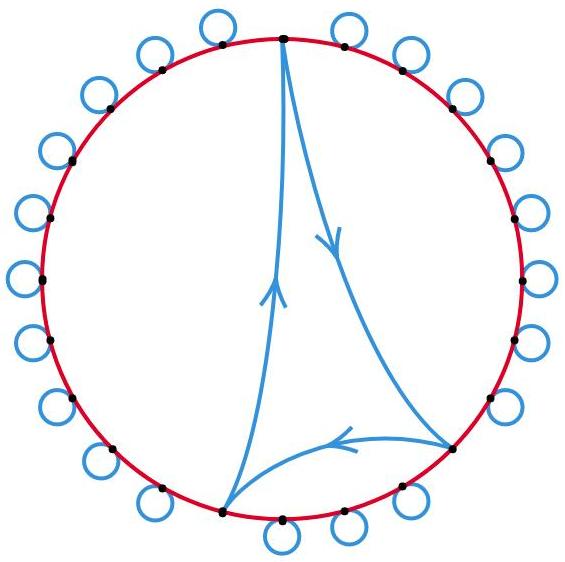}
\caption{A Schreier graph of $\Sym(24)$, 
generated by a $24$-cycle and a $3$-cycle. 
Orientations of red edges are suppressed. 
Note that removing a vertex from the interior 
either of the two long red arcs (and shortening the arc) 
does not change the set of isomorphism classes of balls 
of radius $3$. Proceeding in this way, 
we can produce a Schreier graph with as few as $18$ vertices 
which is locally isomorphic at radius $3$.}
\end{figure}

A similar argument works for permutation groups generated by 
an $n$-cycle and a $k$-cycle, with balls of radius $l$ 
satisfying $n \geq (2l+1)k$, 
though we shall not rehearse the details. 
\end{ex}

\subsection{Bou-Rabee \& Seward's Groups} \label{BoRaSewaSubsect}

In \cite{BoRaSewa}, 
Bou-Rabee and Seward proved the following Theorem. 

\begin{thm} \label{BoRaSewaThm}
For any function $f : \mathbb{N} \rightarrow \mathbb{N}$, 
there is a residually finite group $B_f$ 
and a $2$-element generating set $S_f \subseteq B_f$ such that 
$f(n) \geq \mathcal{R} ^{S_f} _{B_f} (n)$ for all $n \geq 8$. 
\end{thm}

The construction of $B_f$ is based on B.H. Neumann's 
construction of a continuum of pairwise non-isomorphic 
$2$-generated groups \cite{Neum}. 

For $d \in \mathbb{N}$, let $\alpha_d , \beta_d \in \Sym(d)$ 
be given by $\alpha_d =(1 \; 2 \cdots d) , \beta_d = (1 \; 2 \; 3)$. 
If $d$ is odd then$\langle \alpha_d,\beta_d \rangle = \Alt(d)$; 
indeed $\langle \alpha_d ^q,\beta_d \rangle = \Alt(d)$ 
whenever $(d,q)=1$. 
For functions $d,q : \mathbb{N} \rightarrow \mathbb{N}$, 
with $\im(d)$ consisting of odd integers at least $5$, let: 
\begin{equation*}
s = s(d,q) = (\alpha_{d(n)} ^{q(n)})_n ,t= t(d) = (\beta_{d(n)})_n \in G(d) = \prod_{n \in \mathbb{N}} \Alt (d(n))
\end{equation*}
and set $S = S(d,q) =\lbrace s,t\rbrace$; 
$N(d,q) =\langle S \rangle\leq G$. 
The groups studied in \cite{Neum} were the special case 
for which $d$ is strictly increasing and $q$ is identically $1$. 
In this setting Neumann proved the following result. 

\begin{thm}
For all $m \geq 5$, 
$N(d,1)$ has a normal subgroup isomorphic to $\Alt(m)$ 
iff $m \in \im (d)$. 
In particular, 
for any other strictly increasing function 
$d^{\prime} : \mathbb{N} \rightarrow \mathbb{N}_{\geq 5}$, 
$N (d,1) \cong N (d^{\prime},1)$ iff 
$d = d^{\prime}$. 
\end{thm}

\begin{coroll}
There are uncountably many isomorphism types of groups $N(d,1)$. 
\end{coroll}

Now we explain how the groups constructed in Theorem 
\ref{BoRaSewaThm} fit into this framework. 
Without loss of generality suppose that 
$f : \mathbb{N} \rightarrow \mathbb{N}$ 
is an increasing function and $f(n)\geq n$ for all $n\in\mathbb{N}$. 
The technical core of Theorem \ref{BoRaSewaThm} is the following 
number-theoretic observation. 

\begin{propn}[\cite{BoRaSewa} Section 1] \label{BoRaSewaNoProp}
There exist $d,p,q : \mathbb{N} \rightarrow \mathbb{N}$ 
such that for all $n \in \mathbb{N}$: 
\begin{itemize}
\item[(i)] $p : \mathbb{N} \rightarrow \mathbb{N}$ is odd-valued; 
strictly increasing, and $p(1)=1$; 
\item[(ii)] $q(n) \geq 3$ is odd; 
\item[(iii)] $d(n) = p(n)q(n)+2$; 
\item[(iv)] $d(n) \geq 2 f (4p(n+1)+4)$ is odd; 
\item[(v)] For all $i \leq n$, 
$p(n+1)q(i) \not\equiv \pm 1 , \pm 2 \mod d(i)$. 
\end{itemize}
\end{propn}

Moreover the proof of Proposition \ref{BoRaSewaNoProp} 
gives an explicit recursive construction of $d$, $p$ and $q$. 

For a given function $f$, 
let $d,p,q$ be as in the conclusion of 
Proposition \ref{BoRaSewaNoProp}, 
and set $B_f = N(d,q)$, so that $B_f = \langle S(d,q) \rangle$. 
By the conditions of Proposition \ref{BoRaSewaNoProp}, 
$d(n)$ and $q(n)$ are coprime, 
so the projection $\pi_n : B_f \rightarrow \Alt (d(n))$ 
is surjective. 
Theorem \ref{BoRaSewaThm} then follows from the 
next result. 

\begin{propn}[\cite{BoRaSewa} Section 1] \label{BRSkeppropn}
For all $n\in\mathbb{N}$, 
there exists $w_n \in B_{S(d,q)}(4p(n)+4)$ such that for $m\in \mathbb{N}$, 
$\pi_m (w_n) \neq 1$ iff $m=n$. 
\end{propn}

Since $\pi_n (B_f) = \Alt (d(n))$, 
and by the simplicity of $\Alt (d(n))$, 
the normal closure of $w_n$ in $B_f$ 
is a subgroup $H_n$ of $B_f$ isomorphic to $\Alt (d(n))$. 
Now let $m \in \mathbb{N}$ and let 
$Q$ be a finite group and let $\phi : B_f \rightarrow Q$ 
be a homomorphism, such that $\pi \mid_{B_{S(d,q)}(m)}$ is 
injective. Supposing $m \geq 8$, 
let $n \in \mathbb{N}$ be such that $4p(n)+4 \leq m \leq 4p(n+1)+4$. 
Thus $\pi(w_n) \neq 1$, 
so $\pi \mid_{H_n}$ is non-trivial, 
and by simplicity of $H_n$ is injective. 
By Proposition \ref{BoRaSewaNoProp} (iv), 
\begin{equation*}
\lvert Q \rvert \geq \lvert H_n \rvert 
= d(n)! / 2 \geq f(4p(n+1)+4)
\end{equation*}
so $\mathcal{R}_{B_f} ^{S(d,q)} (m) \geq f(4p(n+1)+4) \geq f(m)$. 

We now complete our bound on the LEF growth of $B_f$. 
Combined with Theorem \ref{BoRaSewaThm} above, 
this completes the proof of Theorem \ref{MainThm}. 

\begin{thm} \label{GenNeumUBThm}
For any functions $d,q : \mathbb{N} \rightarrow \mathbb{N}$ satisfying 
$\big( d(n),q(n) \big) = 1$ for all $n \in \mathbb{N}$, 
\begin{equation*}
\mathcal{L}_{N(d,q)} ^{S(d,q)} (l) \preceq \exp (\exp(l)). 
\end{equation*}
\end{thm}

\begin{proof}
There exists $A_l \subseteq \mathbb{N}$ with 
$\lvert A_l \rvert \leq \lvert B_S (3l) \rvert^2 \preceq \exp(l)$ 
such that, for all $w(s,t),v(s,t) \in B_S (3l)$, 
if $w(s,t)\neq v(s,t)$, then: 
\begin{center}
$\pi_k(w(s,t)) = w(s_k,t_k)\neq v(s_k,t_k) = \pi_k(v(s,t))$
\end{center}
for some $k \in A_l$. Then: 
\begin{equation*}
(\prod_{k \in A_l} \pi_k) : B_f \rightarrow \prod_{k \in A_l} \Alt (d(k))
\end{equation*}
is a homomorphism whose restriction to $B_S(3l)$ is injective. 
For $k \in A_l$, let $d_l ^{\prime} (k) = 18l+3$ if 
$d(k) \geq 18l+3$ and $d_l ^{\prime} (k) = d(k)$ otherwise. 
By coprimality of $d(k)$ and $q(k)$, 
$s_k$ is a $d(k)$-cycle. 
Thus by Example \ref{ncyc3cycex}, 
there is a local embedding 
$\phi_{l,k} : B_{\lbrace s_k , t_k \rbrace} (2l) \rightarrow 
\Alt \big( d_l ^{\prime} (k) \big)$. As a result, 
\begin{equation*}
\big(\prod_{k \in A_l} (\phi_{l,k} \circ \pi_k)\big) : 
B_{S_f} (2l) \rightarrow G_l = \prod_{k \in A_l} \Alt (d_l ^{\prime}(k))
\end{equation*}
is a local embedding, and: 
\begin{align*}
\lvert G_l \rvert \leq ((18l+3)!)^{\lvert A_l \rvert}
\preceq (l!)^{\exp(l)} \leq \exp (\exp(l)l\log l) 
\leq \exp(\exp(2l))
\end{align*}
so $\mathcal{L}_{B_f} ^{S_f} (2l) \leq \lvert G_l \rvert 
\preceq \exp (\exp(l))$. 
\end{proof}

\begin{proof}[Proof of Theorem \ref{BfBoundThm}]
Apply the conclusion of Theorem \ref{GenNeumUBThm} 
to $B_f = N(d,q)$ for $d,q$ as in Proposition \ref{BoRaSewaNoProp}. 
\end{proof}

\subsection{Presentations of Neumann's Groups}

There are only countably many finitely presented groups, 
so almost none of the groups $N(d,1)$ are finitely presented. 
By comparing homomorphisms and partial homomorphisms from $N(d,1)$ 
to finite symmetric groups, 
much in the style of the proof of Theorem \ref{MainThm}, 
we show that in fact none of them are. 
Two proofs of this result already appear in \cite{BaumMill}: 
one using HNN extensions; the other based on homological considerations. 

We let $s,t\in N(d,1)$ and $\alpha_d , \beta_d \in \Sym(d)$ 
be as in Subsection \ref{BoRaSewaSubsect} 
and continue to write $S=\lbrace s,t\rbrace$. 

\begin{lem} \label{NeumQuotLem}
There is an epimorphism $\phi: N(d,1) \rightarrow B_{\infty}$, 
where $B_{\infty}$ is a group admitting a short exact sequence 
\begin{equation*}
1 \rightarrow A_{\infty} \rightarrow B_{\infty} \rightarrow \mathbb{Z} \rightarrow 1
\end{equation*}
(here $A_{\infty}$ is the group of finitely supported even permutations 
of the set $\mathbb{Z}$, an infinite simple group). 
Moreover $\ker (\phi) = \bigoplus_{n \in \mathbb{N}} \Alt (d(n))$. 
\end{lem}

\begin{proof}
This is sketched in Problem 35 from Chapter III.B of \cite{delaHar}. 
\end{proof}

\begin{propn} \label{NeumQuotProp}
For $m \geq 4$ even, $N(d,1)$ admits no quotient isomorphic to $\Sym(m)$. 
\end{propn}

\begin{proof} 
Suppose to the contrary that we have an epimorphism 
$\pi: N(d,1) \rightarrow \Sym(m)$. 
We apply Lemma \ref{NeumQuotLem}: 
$\bigoplus \Alt(d(n))$ lies in the kernel of $\pi$, 
for if not, $\Sym(m)$ would have some $\Alt(d(n))$ as a composition factor 
(since each $\Alt(d(n))$ is normal in $N(d,1)$), 
and this is not the case, as $d(n)$ is odd. 
Thus $\pi$ descends to an epimorphism 
$\overline{\pi} : B_{\infty} \rightarrow \Sym(m)$. 
$A_{\infty} \vartriangleleft B_{\infty}$ 
is infinite simple 
so $\overline{\pi}$ descends 
to an epimorphism $\mathbb{Z} \rightarrow \Sym(m)$, a contradiction. 
\end{proof}

\begin{propn} \label{NeumPHProp}
For all $l \in \mathbb{N}$ there exists $m_0 \in \mathbb{N}$ 
such that for all $m \geq m_0$ even, 
there is a partial homomorphism 
$\pi_m : B_{\lbrace s,t \rbrace} (l) \rightarrow \Sym(m)$ 
extending $s \mapsto \alpha_m$, $t \mapsto \beta_m$, 
whose image generates $\Sym(m)$. 
\end{propn}

\begin{proof}
This follows from the argument in Example \ref{ncyc3cycex}. 
Set $m_0=6l+3$, $m^{\prime}\geq m\geq m_0$ with $m^{\prime}\in\im(d)$, 
and let $\rho : N(d,1) \rightarrow \Alt(m^{\prime})$ 
be projection. 
Then $\rho (s)=\alpha_{m^{\prime}}$ and $\rho(t)=\beta_{m^{\prime}}$. 
By Example \ref{ncyc3cycex} there is a local embedding 
$\phi:B_{\lbrace \alpha_{m^{\prime}},\beta_{m^{\prime}} \rbrace} (l) 
\rightarrow \Sym(m)$ extending 
$\alpha_{m^{\prime}} \mapsto \alpha_m$, 
$\beta_{m^{\prime}} \mapsto \beta_m$. 
The composition $\phi \circ \rho|_{B_S(l)}$ 
will be the desired partial homomorphism; 
it suffices to show that its image generates $\Sym(m)$. 
Since $\beta_m$ is a $3$-cycle, by Jordan's Theorem on permutation groups, 
we need only show that $\langle \alpha_m , \beta_m \rangle$ is primitive. 
Consider a block-system on $\lbrace 1,\ldots,m\rbrace$ 
consisting of blocks of size $\geq 2$; 
the action of $\beta_m$ forces $1,2,3$ to lie in the same block $B$. 
The action of $\alpha_m$ then shows that if $k,k+1 \in B$, 
then $k+2 \in B$ also, hence $B$ is the only block. 
\end{proof}

\begin{proof}[Proof of Theorem \ref{NeumFPThm}]
Suppose $N(d,1) = \langle s,t \rangle$ admits a finite presentation 
$\langle s,t \mid R \rangle$. 
Then there exists $l_0 \in \mathbb{N}$ such that 
$R \subseteq B_{\lbrace s,t \rbrace} (l_0)$ (in $F(s,t)$). 
By Proposition \ref{NeumPHProp}, 
there is a partial homomorphism 
$\pi_m : B_{\lbrace s,t \rbrace} (l_0) \rightarrow \Sym(m)$ 
for all sufficiently large even $m$, whose image generates $\Sym(m)$. 
By Lemma \ref{LEFRFlem}, $\pi_m$ extends to an epimorphism 
$N(d,1) \rightarrow \Sym(m)$, 
contradicting Proposition \ref{NeumQuotProp}. 
\end{proof}

\section{Metric approximations of groups} \label{MetricApproxSect}

There is a wide range of approximation properties strictly weaker than LEF. 
Many of these involve a choice of bi-invariant metric on the approximating groups: 
given a finite subset $F$ of the group $\Gamma$ we have, 
instead of an injective partial homomorphism $\pi : F \rightarrow Q$ 
to a finite group $Q$, an \emph{approximation} $\pi : F \rightarrow Q$ 
of $F$ into a finite (or compact) metric group $(Q,d)$. 
Injectivity of $\pi$ is replaced by the condition that distinct points of $F$ 
are mapped to elements of $Q$ far apart with respect to $d$, 
and the partial homomorphism property 
is replaced by the requirement that whenever $g,h,gh \in F$, 
$d\big( \pi(gh),\pi(g)\pi(h) \big)$ is small. 

Let $\Gamma$ be finitely generated by the set $S$. 
A general framework for quantifying metric approximations of finitely generated groups 
is described in \cite{ArzChe}, wherein are introduced, in particular, 
functions $\mathcal{D}_{\Gamma,S} ^{\sof}$, $\mathcal{D}_{\Gamma,S} ^{\lin}$, 
$\mathcal{D}_{\Gamma,S} ^{\hyp}$, $\mathcal{D}_{\Gamma,S} ^{\fin}$ 
which quantify the properties of \emph{soficity}, 
\emph{linear soficity} (over a fixed field $\mathbb{K}$), 
\emph{hyperlinearity} and \emph{weak soficity}, respectively. 
As before, these functions are independent of $S$ up to $\approx$. 
We do not require any features of these functions other than the facts discussed below, 
so we do not define them here. 
Each of these growth functions $\mathcal{D}_{\Gamma,S} ^{\ast}$ satisfy 
$\mathcal{D}_{\Gamma,S} ^{\ast}(n) \leq \mathcal{R}_{\Gamma} ^S (n)$
for all $n \in \mathbb{N}$ (\cite{ArzChe} Section 4.3), 
and one may ask how far apart the growth can be among residually finite groups $\Gamma$. 
In the case of the weakly sofic growth function, 
this was the subject of Question 66 of \cite{ArzChe}. 
We use Theorem \ref{MainThm} to resolve this question. 

\begin{propn} \label{MetricIneqProp}
For all $n \in \mathbb{N}$, 
\begin{itemize}
\item[(i)] $ \mathcal{D}_{\Gamma,S} ^{\lin}(n) \leq \mathcal{D}_{\Gamma,S} ^{\sof} (n)$; 
\item[(ii)] $\mathcal{D}_{\Gamma,S} ^{\hyp}(n)\leq \mathcal{D}_{\Gamma,S} ^{\sof}(2n^2)$; 
\item[(iii)] $ \mathcal{D}_{\Gamma,S} ^{\fin}(n) \leq \mathcal{D}_{\Gamma,S} ^{\sof}(n)!$. 
\end{itemize}
\end{propn}

\begin{proof}
See respectively p.13 (following Definition 18); p.10 (following Definition 14) 
and p.14 (following Definition 21) of \cite{ArzChe} for the inequalities (i)-(iii). 
\end{proof}

If $Q$ be a finite group and 
there is a local embedding $\pi : B_S (n) \rightarrow Q$ then, 
letting $\rho : Q \rightarrow \Sym(Q)$ be the regular representation, 
$\rho \circ \pi : B_S (n) \rightarrow \Sym(Q)$ defines 
a $(n,1)$-approximation of $\Gamma$ in the metric group $(\Sym(Q),d_{\Ham})$ 
in the sense of \cite{ArzChe} Definition 1. 
We therefore have the following, also noted in \cite{ArzChe} Section 4.3. 

\begin{lem} \label{LEFMetricLem}
For all $n\in\mathbb{N}$, 
$\mathcal{D}_{\Gamma,S} ^{\sof} (n) \leq \mathcal{L}_{\Gamma} ^S (n)$. 
\end{lem}

\begin{thm} \label{MetricGapThm}
For any $\ast \in \lbrace \sof,\lin,\hyp,\fin \rbrace$, 
there exists an increasing function $f^{\ast} : \mathbb{N} \rightarrow \mathbb{N}$ 
such that, for any increasing function $f:\mathbb{N} \rightarrow \mathbb{N}$, 
there exists a finitely generated residually finite group $B_f$ 
such that $f \preceq \mathcal{R}_{B_f}$ but 
$\mathcal{D}_{B_f} ^{\ast} \preceq f^{\ast}$. 
In particular, there exists a finitely generated residually finite group $\Gamma$ 
such that $\mathcal{D}_{\Gamma} ^{\ast} \precnapprox \mathcal{R}_{\Gamma}$. 
\end{thm}

\begin{proof}
Let $B_f$ be as in Theorem \ref{MainThm} (i), 
so the bound on $\mathcal{R}_{B_f}$ holds. 
The upper bound on $\mathcal{D}_{B_f} ^{\ast}$ 
follows by combining the inequalities from Proposition \ref{MetricIneqProp}; 
Lemma \ref{LEFMetricLem} and Theorem \ref{MainThm} (ii). 
The final claim holds for $\Gamma \cong B_f$, 
for any sufficiently quickly growing function $f$. 
\end{proof}

\section{A universal upper bound} \label{univubsect}

The next Proposition was pointed out to the author by A. Thom. 

\begin{propn} \label{univubprop}
For all $D,n  \in \mathbb{N}$ the set: 
\begin{center}
$L_{D,n} = \lbrace \mathcal{L}_{\Gamma} ^S (n) : \Gamma = \langle S\rangle \text{ LEF, }\lvert S \rvert \leq D \rbrace$
\end{center}
is bounded above. 
\end{propn}

\begin{proof}
We use the perspective of Subsection \ref{markedsubsect}. 
Let $B_{D,n}$ be the set of marked balls of radius $n$ 
and valence at most $2D$ occuring in Cayley graphs of all 
$D$-generated groups. 
It is clear that $B_{D,n}$ is finite for every $D$ and $n$. 
Let $B^f _{D,n} \subseteq B_{D,n}$ be the set of those $n$-balls 
which occur in a Cayley graph of some $D$-generated finite group. 
For each $b \in B^f _{D,n}$ let $C(b)$ 
be the (nonempty) set of finite groups admitting marked Cayley graphs 
whose $n$-ball is isomorphic to $b$. 

Let $1 \leq d \leq D$, and let $\Gamma$, $\mathbf{a}$ 
and $\Phi_{\Gamma} ^{\mathbf{a}}(n)$ be as in 
Proposition \ref{LEFmarkedgrps}. 
Since $\Gamma$ is LEF, $B_{\mathbf{a}} (n) \in B_{D,n} ^f$, so: 
\begin{equation*}
\Phi_{\Gamma} ^{\mathbf{a}}(n) \leq 
\max_{b \in B^f _{D,n}} \min_{Q \in C(b)} \lvert Q \rvert, 
\end{equation*}
the right-hand side of which is independent of 
$(\Gamma,\mathbf{a})$. 
The result follows from Proposition \ref{LEFmarkedgrps}. 
\end{proof}

We may ask, for fixed $D$, how quickly an upper bound for 
$L_{D,n}$ need grow as a function of $n$. 
As it transpires, the growth is very fast indeed. 

\begin{thm} \label{norecbdthm}
Let $L_{D,n}$ as in Proposition \ref{univubprop}. 
Let $\Theta_D (n) = \max (L_{D,n})$. 
Then for $D \geq 2$, $\Theta_D$ is not bounded above 
by any recursive function. 
\end{thm}

From Proposition \ref{univubprop} and Theorem \ref{norecbdthm} 
we immediately have Theorem \ref{UBThm}. 
Theorem \ref{norecbdthm} is obtained by combining the following 
two results. 

\begin{thm}[Wilson, \cite{Wilson} Theorem A] \label{Wilsonthm}
Every countable residually finite group $\Gamma$ may be embedded 
into a $2$-generated residually finite group $\hat{\Gamma}$. 
\end{thm}

\begin{thm}[Kharlampovich-Myasnikov-Sapir, \cite{KMS} Theorem 4.22] \label{KMSthm}
For every recursive function $f: \mathbb{N} \rightarrow \mathbb{N}$, 
there exists a finitely presented residually finite group $\Delta(f)$
with $\mathcal{R}_{\Delta(f)} \succeq f$. 
\end{thm}

\begin{lem} \label{recurselem}
Let $f : \mathbb{N} \rightarrow \mathbb{N}$ and suppose that, 
for every recursive function $g:\mathbb{N} \rightarrow \mathbb{N}$ 
there exists $C = C(g)>0$ such that, for all $n \in \mathbb{N}$, 
$C f (Cn) \geq g(n)$. 
Then $f$ is not bounded above by any recursive function.  
\end{lem}

\begin{proof}
Suppose otherwise, let $\overline{f}:\mathbb{N}\rightarrow \mathbb{N}$  
be recursive with $\overline{f} \geq f$, and define: 
\begin{equation*}
\tilde{f} (n) = n \cdot\max_{1 \leq m \leq n^2} \overline{f} (m),
\end{equation*}
so that $\tilde{f}$ is recursive. 
Set $\tilde{C} = C(\tilde{f})$, so that for all $n$, 
$\tilde{C} f(\tilde{C} n) \geq \tilde{f}(n)$. 
Taking $n > 2 \tilde{C}$ we have: 
\begin{equation*}
f (\tilde{C} n) > 2 \cdot \max_{1 \leq m \leq n^2} f(m) 
\geq 2 f (\tilde{C} n), 
\end{equation*}
a contradiction. 
\end{proof}

\begin{proof}[Proof of Theorem \ref{norecbdthm}]
By construction, $\Theta_D (n) \leq \Theta_{D+1} (n)$ 
for all $D$ and $n$.  
Thus it suffices to prove the result for $\Theta_2(n)$. 
Let $f : \mathbb{N} \rightarrow \mathbb{N}$ be a recursive function. 
The $\Delta (f)$ be as in Theorem \ref{KMSthm}, 
so that $\mathcal{R}_{\Delta(f)} \succeq f$. 
$\Delta (f)$ is finitely presented, 
so by Proposition \ref{LEFRFpropn}, 
$\mathcal{L}_{\Delta(f)} \succeq f$. 
Apply Theorem \ref{Wilsonthm} to $\Delta (f)$ 
to obtain a $2$-generated residually finite group $\hat{\Delta} (f)$, 
with $\Delta (f) \hookrightarrow \hat{\Delta} (f)$. 
Thus $f \preceq \mathcal{L}_{\Delta(f)} 
\preceq \mathcal{L}_{\hat{\Delta} (f)}$ (by Lemma \ref{subgrplem}). 
The conclusion follows from Lemma \ref{recurselem}. 
\end{proof}

\begin{rmrk}
\normalfont
Proposition \ref{univubprop} shows that there is a single function 
which is an upper bound for the LEF growth of all $D$-generated 
LEF groups. 
As noted in the Introduction, 
this is in sharp contradistinction to the case of the full 
residual finiteness growth, 
as is shown by Bou-Rabee and Seward's Theorem \ref{BoRaSewaThm}. 
Indeed this distinction gives us a simple proof of 
Theorem \ref{MainThm} (i) and (ii) (without explicit bounds on $f_0$). 
Let $D \geq 2$ and let 
$\Gamma$ be any $D$-generated residually finite group satisfying 
$\mathcal{R}_{\Gamma} \succnapprox F_D$ 
(for instance let $\Gamma = B_f$ be as in Theorem \ref{BoRaSewaThm} 
with $f \succnapprox F_D$). 
Then $\mathcal{R}_{\Gamma} \succnapprox \mathcal{L}_{\Gamma}$. 
\end{rmrk}

\begin{propn} \label{WordProbProp}
If $\Gamma$ is recursively presented LEF, with solvable word problem, 
then $\mathcal{L}_{\Gamma} ^S$ is a recursive function. 
\end{propn}

\begin{proof}
First, on receiving input $n \in \mathbb{N}$,
construct the ball of radius $n$ in $\Cay(\Gamma,S)$. 
Such an algorithm for $\Gamma$ exists by solubility of the word 
problem. 

Second, enumerate all Cayley graphs of finite groups, 
in ascending order of group order, and compare their balls 
of radius $n$ with that of $\Cay(\Gamma,S)$. 
The first finite group $Q$ in our enumeration which admits 
a Cayley graph whose $n$-ball is colour-isomorphic 
to that of $\Cay(\Gamma,S)$ satisfies 
$\lvert Q \rvert = \mathcal{L}_{\Gamma} ^S (n)$. 
\end{proof}

\section{Further Questions}

As explained in Section \ref{WPSect}, 
the groups constructed in Theorem \ref{WreathUBThm} 
which afford local embeddings 
from finite subsets of $\Delta \wr \Gamma$ have the form 
$P \wr Q$, where $P$ and $Q$ afford local embeddings from corresponding 
finite subsets of $\Delta$ and $\Gamma$, respectively. 
It is known that not every local embedding from a wreath product need 
arise this way, however. For instance, 
Cornulier shows that $\mathbb{Z} \wr \mathbb{Z}$ is locally 
embeddable into the class of finite alternating groups 
(in the sense of Remark \ref{LEFgrwthCrmrk}; 
a proof is available in Subsection 5.4 of \cite{Capr}: 
the argument is reminiscent of our proof of Theorem \ref{BfBoundThm}). 
It is therefore interesting to ask for other constructions 
of local embeddings of wreath products, and to enquire which are the 
most efficient in the sense of LEF growth. 

\begin{prob} \label{WPopenQ}
Determine exactly the LEF growth of some LEF wreath products 
$\Delta \wr \Gamma$ beyond those provided for 
by the conclusions of Theorem \ref{WPMainThm}. 
In particular, determine $\Delta \wr \Gamma$ 
in some cases for which either (a) $\Gamma$ is not ELEF; 
(b) $\Delta$ is infinite or (c) $Z(\Delta) \neq 1$. 
\end{prob}

We turn next to a class of groups arising from topological 
dynamics, which have recently emerged as a source of 
exotic examples in geometric group theory. 
For $X$ a topological space and $T$ a self-homeomorphism of $X$, 
one can define the topological full group $[[T]]$ 
of the system $(X,T)$. 
Grigorchuk and Medynets \cite{GrigMedy} proved that for 
$(X,T)$ a Cantor minimal system, $[[T]]$ is LEF. 
In this setting Matui \cite{Matu} proved that 
the derived subgroup $[[T]]^{\prime}$ 
is finitely generated iff $(X,T)$ is topologically isomorphic 
to a minimal subshift over a finite alphabet. 
Under these conditions we pose the following. 

\begin{qu}
What is the LEF growth of $[[T]]^{\prime}$? 
How does it depend on the dynamical properties of $(X,T)$? 
\end{qu}

We may also ask for examples of finitely generated LEF groups 
exhibiting LEF growth functions beyond those observed 
in Theorem \ref{GrowthSpectrumThm}. 
A particular challenge would be to look for a 
group of \emph{intermediate} LEF growth.

\begin{qu} \label{intgrthQ}
Does there exist a finitely generated group $\Gamma$ 
such that $\mathcal{L}_{\Gamma}(n) \nsucceq \exp(n)$ 
and for all $C>0$, $\mathcal{L}_{\Gamma}(n) \npreceq n^C$? 
\end{qu}

By Proposition \ref{polyLEFnilpprop} 
and Remark \ref{growthlowerbdrmrk}, 
a group $\Gamma$ satisfying the conditions of Question \ref{intgrthQ} 
has intermediate word growth. 
To this end, it would be interesting to investigate 
the LEF growth of some groups of intermediate growth, 
including those introduced in \cite{Grig}, 
whose local embeddings were studied in that paper. 

\begin{prob}
Let $G_{\omega}$ be one of the finitely generated groups 
of intermediate growth from \cite{Grig}, 
where $\omega \in \lbrace 0,1,2 \rbrace^{\mathbb{N}}$ 
is a sequence which is not eventually constant. 
Estimate the LEF growth of $G_{\omega}$. 
\end{prob}

There are uncountably many isomorphism classes of 
finitely generated LEF groups. 
This may be seen either from the construction of Neumann 
described in Section \ref{MainThmSect} or from the 
results of \cite{Grig}. 
Applying to the groups obtained from these 
constructions a wreath product with a finite perfect group, 
it is not difficult to see that there are uncountably many 
isomorphism classes of finitely generated LEF groups 
which are not RF. 
However as yet we are not able to distinguish 
these groups from one another by their LEF growth. 
This leads to our next question. 

\begin{qu}
Does there exist an uncountable family $\mathcal{G}$ 
of pairwise non-isomorphic finitely generated LEF groups, 
such that for all $\Gamma,\Delta \in \mathcal{G}$, 
if $\mathcal{L}_{\Gamma} \approx \mathcal{L}_{\Delta}$ 
then $\Gamma \cong \Delta$? 
\end{qu}

Next, it is not known whether infinite free Burnside groups are LEF. 
Indeed it is known that either free Burnside groups 
of sufficiently large odd exponent 
are locally embeddable into the class $\mathcal{S}$ 
of nonabelian finite simple groups, 
or there is a non-residually finite hyperbolic group
(see Subsection 5.4 of \cite{Capr} for a discussion). 
Short of showing that these groups do not 
locally embed into $\mathcal{S}$ 
(and thereby demonstrating the existence of a non-residually finite 
hyperbolic group), one may ask for the following. 

\begin{prob}
Let $\Gamma$ be an infinite Burnside group 
of sufficiently large exponent. 
Prove a substantial lower bound on 
$\mathcal{L}_{\Gamma,\mathcal{S}} (n)$. 
\end{prob}

See also \cite{ArzChe} Question 53 and Conjecture 54 
for more on the connection between residual finiteness 
of hyperbolic groups and other embedding conditions. 

Regarding metric approximations, in Theorem \ref{MetricGapThm} 
we produced examples of groups with large full residual finiteness growth 
and small sofic profile by using LEF growth to bound the sofic profile. 
This begs the question of whether sofic profile can still be small in LEF groups 
with large LEF growth. 

\begin{qu}
Does there exist a finitely generated LEF group $\Gamma$ with 
$\mathcal{D}_{\Gamma} ^{\sof} \precnapprox \mathcal{L}_{\Gamma}$? 
How far apart can the growth of $\mathcal{D}_{\Gamma} ^{\sof}$ 
and $\mathcal{L}_{\Gamma}$ be? 
\end{qu}

Finally we can ask for more detailed quantitative information about 
the geometry and topology of the space of LEF groups 
inside the space $\mathcal{G}_d$ of marked $d$-generated groups. 
One consequence of the observations of Section \ref{univubsect} 
is that an $\epsilon$-net for this set must contain a 
(marking of a) group of order at least comparable to $f(1/\epsilon)$ 
for any recursive function $f$. 
In another direction we can ask for an estimate of the minimal 
size of an $\epsilon$-net, and how this compares 
to the minimal size of a net for $\mathcal{G}_d$. 

\begin{qu}
Let $f_d (n)$ and $g_d (n)$ be 
the number of marked balls of radius $n$ in 
all $d$ generated groups and all finite $d$-generated groups, 
respectively. How do $f_d (n)$ and $g_d (n)$ grow 
as functions of $n$? 
Is $g_d (n)/f_d (n)$ bounded away from zero? 
\end{qu}

An easy combinatorial argument shows that 
$f_d (n) = \exp(\exp(O_d(n)))$. 
If $\Gamma$ is an infinite group 
with no non-trivial finite quotients, 
which admits a finite presentation with $d$ 
generators and all relations of length at most $r$, 
then by Proposition \ref{LEFRFpropn} 
any ball of radius at least $r$ in the corresponding 
Cayley graph does not occur as a ball in the 
Cayley graph of any finite group. 
Examples of such groups $\Gamma$ 
include the Higman group (for which we may take $(d,r)=(4,5)$). 
Beyond this it seems far from clear at present 
what one should expect the behaviour of 
$f_d (n)$ and $g_d (n)$ to be. 

\subsection*{Acknowledgements}

This work was supported by ERC grant no. 648329 ``GRANT''. 
I am grateful to Vadim Alexeev, Goulnara Arzhantseva, Khalid Bou-Rabee, 
Masato Mimura and 
Andreas Thom for enlightening conversations which helped to 
shape my thinking on the subject of this paper. 
Parts of this work were undertaken while the author was 
a visiting fellow at the Hausdorff Research Institute for Mathematics 
in Bonn for the program ``Logic and Algorithms in 
Group Theory'', 
and other parts while attending the program ``$L^2$-invariants and 
their analogues in positive characteristic'' at 
the Instituto de Ciencias Matem\'{a}ticas in Madrid. 
It is my pleasure to thank the organizers 
for putting together these excellent programs, 
and HIM and ICMAT for providing pleasant working conditions.

\end{document}